\documentclass[12pt,reqno,draft]{amsart}
\usepackage{amsthm}
\usepackage{amscd}
\usepackage{amsfonts}
\usepackage{amssymb}
\usepackage{amsgen}
\usepackage{amsmath}
\usepackage{amsopn}
\usepackage{verbatim}
\usepackage{xypic}
\usepackage{xspace}
\usepackage{multicol}
\usepackage{url}
\usepackage{upref}

\theoremstyle{plain}
\newtheorem{thm}{Theorem}[section]
\newtheorem{lem}[thm]{Lemma}

\theoremstyle{definition}

\theoremstyle{remark}
\newtheorem{rem}[thm]{Remark}

 \font\cyr=wncyr10
 \newcommand{\nc}{\newcommand}

\DeclareMathOperator{\sgn}{{sgn}}
\DeclareMathOperator{\sha}{{\mbox{\cyr x}}}

 \nc{\genf}{\genfrac{(}{)}{0pt}{}}
 \def\sqfr#1#2{{\frac{\genf{#1}{#2}}{\genf{#1+#2}{#2}}}}
 \nc{\calH}{{\mathcal H}}
 \nc{\calh}{{\hbar}}

\nc{\per}[1]{\underset{#1}{\boldsymbol \pi}\,}
 \nc{\hyf}{\text{-}}
 \nc{\wdt}{\widetilde}
 \nc{\bt}{{\bf 2}}
 \nc{\id}{{\rm id}}
 \nc{\ot}{\otimes}
 \nc{\Z}{{\mathbb Z}}
 \nc{\R}{{\mathbb R}}
 \nc{\N}{{\mathbb N}}
 \nc{\Q}{{\mathbb Q}}
 \nc{\C}{{\mathbb C}}
 \nc{\oI}{{\overline{I}}}
 \nc{\oJ}{{\overline{J}}}
 \nc{\ol}{\overline}
 \nc{\oll}[1]{\underline{#1}}
 \nc{\oK}{{\overline{K}}}
 \nc{\bI}{{\bar{I}}}
 \nc{\bJ}{{\bar{J}}}
 \nc{\bK}{{\bar{K}}}

 \nc{\ds}{{\scriptstyle\diamondsuit}}
 \nc{\ga}{\alpha}
 \nc{\gk}{\kappa}
 \nc{\gs}{\sigma}
 \nc{\gl}{\lambda}

 \nc{\bfone}{{\bf 1}}
 \nc{\bfA}{{\bf A}}
 \nc{\bfe}{{\bf e}}
 \nc{\bfp}{{\bf p}}
 \nc{\bfq}{{\bf q}}
 \nc{\bfs}{{\bf s}}
 \nc{\tbfs}{{\tilde{\bf s}}}
 \nc{\tbft}{{\tilde{\bf t}}}
 \nc{\tbfu}{{\tilde{\bf u}}}
 \nc{\bft}{{\bf t}}
 \nc{\bfu}{{\bf u}}
 \nc{\bfv}{{\bf v}}
 \nc{\bfw}{{\bf w}}
 \nc{\bfx}{{\bf x}}
 \nc{\bfy}{{\bf y}}
 \nc{\frakS}{{\mathfrak S}}
\begin{document}

\title[Multiple Harmonic Sums and Multiple Zeta Star Values]{Identity Families of Multiple Harmonic Sums and \\
Multiple Zeta Star Values}

\author{Jianqiang Zhao}
\date{November 28, 2013}

\maketitle

\begin{center}
\vskip-.7cm
Department of Mathematics, Eckerd College, St. Petersburg, FL 33711\\
Email: zhaoj@eckerd.edu
\end{center}

\begin{quote}
{\small\noindent
\textbf{Abstract.} In this paper we present many new families of identities for multiple harmonic sums using binomial coefficients. Some of these generalize a few recent results of Hessami Pilehrood, Hessami Pilehrood and Tauraso \cite{HessamiPilehrood2Ta2013}. As applications we prove several conjectures involving multiple zeta star values (MZSV): the Two-one formula conjectured by Ohno and Zudilin, and a few conjectures of Imatomi et al. involving $\bt$-$3$-$\bt$-$1$ type MZSV, where the boldfaced $\bt$ means some finite string of 2's.}
\end{quote}

\section{Introduction}\label{intro}
For over two hundred years, Euler's pioneering work on double zeta values \cite{Euler1775} was largely neglected, until in the early 1990s when Zagier showed the importance of the more general multiple zeta values in his famous paper \cite{Zagier1994}. Since then these values have come up in many areas of current research in mathematics and physics, such as knot theory, motivic theory, mirror symmetry and Feynman integrals, to name just a few. One of the central problems is to determine various $\Q$-linear relations among these values, many of which have been discovered numerically first and then proved rigorously later. One such family that still defies a proof until now is the celebrated Two-one formula discovered by Ohno and Zudilin \cite{OhnoZu2008}.

The main goal of this paper is to give a comprehensive study of multiple zeta star values of a few special types using the corresponding identities established first for multiple harmonic sums. As one of the applications, we give a concise proof of the Two-one formula.

We now recall some definitions. Let $\N$ be the set of natural numbers, $\Z$ the set of integers, $\Z^*$ the set of nonzero integers, and $\N_0=\N\cup\{0\}$. For any $\ell\in \N$
and $\bfs=(s_1, s_2, \ldots, s_\ell)\in (\Z^*)^\ell$ we define the \emph{(alternating) multiple harmonic sum} (MHS for short)
{\allowdisplaybreaks
\begin{align}\label{equ:Hdefn}
H_n(s_1, s_2, \ldots, s_\ell)=&\sum_{n\ge k_1>k_2>\ldots>k_\ell\ge 1}\prod_{i=1}^\ell
\frac{\sgn (s_i)^{k_i}}{k_i^{|s_i|}},\\
 \label{equ:Sdefn}
H^\star_n(s_1, s_2, \ldots, s_\ell)=&\sum_{n\ge k_1\ge k_2\ge\ldots\ge k_\ell\ge 1}\prod_{i=1}^\ell
\frac{\sgn (s_i)^{k_i}}{k_i^{|s_i|}}.
\end{align}}
This star-version has been denoted by $S_n$ in the literature but it seems to be more appropriate
to use $H^\star$ in this paper due to its close connection with multiple zeta star values
to be defined momentarily.
Conventionally, we call $l({\bf s}):=\ell$ the depth and
$|\bfs|:=\sum_{i=1}^\ell |s_i|$ the weight.
For convenience we set  $H_n(\bfs)=0$ if $n<l(\bfs)$,
$H_n(\emptyset)=H^\star_n(\emptyset)=1$ for all $n\ge 0$, and $\{s_1, s_2, \ldots, s_\ell\}^r$
the set formed by repeating the composition $(s_1, s_2, \ldots, s_\ell)$ exactly $r$ times.

When $\bfs=(s_1, s_2, \ldots, s_\ell)\in  (\Z^*)^\ell$ with $(s_1,\sgn(s_1))\ne (1,1)$
we set, respectively, the \emph{(alternating) Euler sum} and
the \emph{(alternating) star Euler sum} by
\begin{equation}\label{equ:EulerSumDefn}
\zeta(\bfs)=\lim_{n\to\infty} H_n(\bfs),\qquad
\zeta^\star(\bfs)=\lim_{n\to\infty} H^\star_n(\bfs).
\end{equation}
When $\bfs\in\N^\ell$ they are called the \emph{multiple zeta value} (MZV)
and the \emph{multiple zeta star value} (MZSV), respectively.

Notice that we have abused the notation in the definitions \eqref{equ:Hdefn}, \eqref{equ:Sdefn} and \eqref{equ:EulerSumDefn} since all these sums can be evaluated at negative integers,
because, for example, a multiple zeta function can be analytically continued to the whole complex space and it is possible to define its values at negative integers \cite{AET2001,Zhao1999}. Hence throughout the
paper we will write $\bar n$ whenever $-n$ appears as an argument. For instance,
we really should write $\zeta(\ol{2})$ instead of $\zeta(-2)$ which usually means
the Riemann zeta function at $-2$.

We now state the Two-one formula conjectured by Ohno and Zudilin \cite{OhnoZu2008}.
\begin{thm}\label{thm:MZSV21}
Let $r\in \N$ and $\bfs=(\{2\}^{a_1},1,\dots,\{2\}^{a_r},1)$
where $a_1\in\N$ and $a_j\in\N_0$ for all $j\ge 2$. Then we have
 \begin{equation*}
\zeta^\star(\bfs) =\sum_{\bfp} 2^{\ell(\bfp)}\zeta(\bfp),
\end{equation*}
where $\bfp$ runs through all indices of the form $(2a_1+1)\circ \cdots\circ (2a_r+1)$
with ``$\circ$'' being either the symbol ``,'' or the symbol ``$+$''.
\end{thm}

Until recently, not many nontrivial families of identities relating MZVs or MZSVs with truly alternating Euler sums have been proved. One such result is proved by Zlobin \cite{Zlobin2005}
\begin{equation}\label{equ:Zlobin}
 \zeta^\star(\{2\}^{n}) = -2\zeta(\ol{2n}) \quad \text{for all }n\ge 1.
\end{equation}
Another is proved in \cite{Zhao2010a}: $\zeta(\{3\}^{n}) = 8^n\zeta(\{\ol{2},1\}^n)$. Recently, two more appear as (27) and (28) of \cite{HessamiPilehrood2Ta2013} one of which yields a new proof of an identity of \cite{Zagier2012}. Notice that \cite[(22)]{HessamiPilehrood2Ta2013} implies \eqref{equ:Zlobin} easily (see Lemma~\ref{lem:limitBound}). To provide more such families in this paper we need some book-keeping first. A boldface of a single digit number means the number is repeated a few times. We underline a string pattern to mean the whole pattern is repeated. Thus the Two-one formula should be written as $\oll{\bt\hyf1}$ formula and in each repetition the $\bt$'s may have different lengths.

Besides the  $\oll{\bt\hyf1}$ formula in Theorem~\ref{thm:MZSV21} we show many analogous formulas in this paper. For ease of
reference we list them as follows:
\begin{enumerate}
  \item \label{case:21}
  $\oll{\bt\hyf1}$: \S\ref{sec:MHS21&212} for MHS, \S\ref{sec:MZSV21&212proof} for MZSV;
  \item  \label{case:212}
$\oll{\bt\hyf1}\hyf\bt$ (nontrivial substring $\bt$ at the end): \S\ref{sec:MHS21&212} for MHS, \S\ref{sec:MZSV21&212proof} for MZSV;
  \item  \label{case:2c2}
$\oll{\bt\hyf c}\hyf\bt$ ($c\ge 3$ and $\bt$ at the end may be trivial):
    to appear in a joint work with my student Erin Linebarger \cite{LinebargerZh2013};
  \item  \label{case:2c21}
$\oll{\bt\hyf c\hyf \bt\hyf1}$: \S\ref{sec:MHSMZSV2c21&212c21} for MHS and MZSV;
  \item  \label{case:2c212}
$\oll{\bt\hyf c\hyf \bt\hyf1}\hyf\bt$ (nontrivial $\bt$ at the end): \S\ref{sec:MHSMZSV2c212&212c212} for MHS and MZSV;
  \item  \label{case:212c21}
$\bt\hyf1\hyf\oll{\bt\hyf c\hyf\bt\hyf1}$: \S\ref{sec:MHSMZSV2c21&212c21} for MHS and MZSV;
  \item  \label{case:212c212}
$\bt\hyf1\hyf\oll{\bt\hyf c\hyf\bt\hyf1}\hyf\bt$ (nontrivial $\bt$ at the end): \S\ref{sec:MHSMZSV2c212&212c212} for MHS and MZSV;
  \item \label{case:212c2}
$\oll{\bt\hyf1\hyf\bt\hyf c}\hyf\bt$ ($\bt$ at the end may be trivial): \S\ref{sec:MHSMZSV212c2&2c212c2} for MHS and MZSV;
  \item \label{case:2c212c2}
$\bt\hyf c\hyf \oll{\bt\hyf1\hyf\bt\hyf c}\hyf\bt$ ($\bt$ at the end may be trivial): \S\ref{sec:MHSMZSV212c2&2c212c2} for MHS and MZSV;
  \item \label{case:1c1}
  $\oll{\bfone\hyf c}\hyf\bfone$ ($c\ge 1$ and $\bfone$ at the end may be trivial):  \S\ref{sec:MHS1c&1c1} for MHS.
\end{enumerate}
For example, the following is the $\oll{\bt\hyf1}\hyf\bt$ \emph{formula}.
\begin{thm}\label{thm:MZSV212}
Let $r\in \N$ and $\bfs=(\{2\}^{a_1},1,\dots,\{2\}^{a_r},1,\{2\}^{a_{r+1}})$
where $a_1,a_{r+1}\in\N$ and $a_j\in\N_0$ for all $2\le j\le r$. Then we have
 \begin{equation*}
\zeta^\star(\bfs) =-\sum_{\bfp} 2^{\ell(\bfp)}\zeta(\bfp),
\end{equation*}
where $\bfp$ runs through all indices of the form $(2a_1+1)\circ \cdots\circ (2a_r+1)\circ \ol{2a_{r+1}}$
with ``$\circ$'' being either the symbol ``,'' or the symbol O-plus ``$\oplus$''
defined by
\begin{equation}\label{eqn:Oplus}
 a\oplus b=\sgn(a)b+\sgn(b)a  \text{ for all }a,b\in \Z^*.
\end{equation}
\end{thm}

When $r=2$, we have checked numerically the following identities for all $0\le a,b,c\le 2$ and
$ac\ne 0$ with the help of EZ-face \cite{EZface}:
\begin{align*}
\zeta^\star(\{2\}^a,1,\{2\}^b,1,\{2\}^c)=&-2\zeta(\overline{2(a+b+c)+2})
-4\zeta(2a+2+2b,\overline{2c})\\
&-4\zeta(2a+1,\overline{2b+1+2c})-8\zeta(2a+1,2b+1,\overline{2c}).
\end{align*}

One of the main results contained in \cite[Theorem~2.3]{HessamiPilehrood2Ta2013} is the following theorem.
\begin{thm} \label{thm:HP2T-MHSid21+212} {\rm (\cite[Theorem~2.3]{HessamiPilehrood2Ta2013})}
Let $a\in \N_0$ and $b\in\N$. Then for any $n\in\N$
\begin{align}\label{equ:HP2T-MHSid21}
H^\star_n(\{2\}^a,1)&=2\sum_{k=1}^n\frac{\binom{n}{k}}{k^{2a+1}\binom{n+k}{k}}, \tag{A}\\
H^\star_n(\{2\}^a,1,\{2\}^b)&=-2\sum_{k=1}^n\frac{(-1)^{k}\binom{n}{k}}{k^{2a+1+2b}\binom{n+k}{k}}-
4\sum_{k=1}^n\frac{H_{k-1}(\ol{2b})\binom{n}{k}}{k^{2a+1}\binom{n+k}{k}}. \tag{B} \label{equ:HP2T-MHSid212}
\end{align}
\end{thm}

We want to caution the reader that the convention of index ordering in \cite{HessamiPilehrood2Ta2013} is opposite to ours in the definitions \eqref{equ:Hdefn} and \eqref{equ:Sdefn} of MHS. This is the reason why $a$ and $b$ in Theorem~\ref{thm:HP2T-MHSid21+212}(B) is switched from the original statement in \cite[Theorem~2.3]{HessamiPilehrood2Ta2013}.

Theorems~\ref{thm:MHS21} and \ref{thm:MHS212} generalize Theorem~\ref{thm:HP2T-MHSid21+212}(A) and (B) respectively by allowing the arguments to contain an arbitrary number of 2-strings, which lead to the $\oll{\bt\hyf1}$ and $\oll{\bt\hyf1}\hyf\bt$ formulas for MHS. The proofs are straight-forward, however, the difficult part is the discovery of the theorems (using a lot of Maple experiments). By taking limits in these two theorems so that MHS become MZSV we can prove the $\oll{\bt\hyf1}$ formula and the  $\oll{\bt\hyf1}\hyf\bt$  formula for MZSV in Theorems~\ref{thm:MZSV21} and \ref{thm:MZSV212}, respectively.

In \S\ref{sec:ITTWConj} and \S\ref{sec:ITTWConj2} we provide new and concise proofs of a few conjectures first formulated by Imatomi et al.\ in \cite{ITTW2009} concerning MZSV of types $\oll{\bt\hyf 3\hyf\bt\hyf 1}$ and $\oll{\bt\hyf 3\hyf\bt\hyf 1}\hyf\bt$.

In the last section we propose a few possible future research directions, one of which will be carried out in
a sequel to this work in which we will study congruence properties of MHS as further applications of the results we have  obtained in this paper.

\medskip

\noindent{\bf Acknowledgement.} We would like to thank Roberto Tauraso for sending us their preprint \cite{HessamiPilehrood2Ta2013}. This work is partially supported by NSF grant DMS1162116 which enables me to work with my students more productively. In particular, this paper is inspired by a recent collaboration with one of my students, Erin Linebarger \cite{LinebargerZh2013}. We are able to generalize Theorems~2.1 of \cite{HessamiPilehrood2Ta2013} to arbitrary number of strings of 2's using similar ideas contained in this paper.

\section{MHS identities: $\oll{\bt\hyf1}$ formula and $\oll{\bt\hyf1}\hyf\bt$ formula}\label{sec:MHS21&212}
To state our main theorems we need some additional notations first. For $\bfs=(s_1,\dots,s_m)\in(\Z^*)^m$ we
define the mollified companion of $H_n(\bfs)$ by
\begin{equation}\label{equ:calHn}
  \calH_n(\bfs):= \sum_{n\ge k_1>\cdots >k_m\ge 1} \sqfr{n}{k_1}\,
 \prod_{j=1}^m  \frac{\sgn(s_j)^{k_j} }{k_j^{|s_j|}}
 =\sum_{k=1}^n \frac{\sgn(s_1)^{k}}{k^{|s_1|}} \sqfr{n}{k}\, H_{k-1}(s_2,\dots,s_m).
\end{equation}
We further define $\Pi(\bfs)$ to be the set of all indices of the form
$(s_1\circ\cdots\circ s_m)$ where ``$\circ$'' being either the symbol ``,'' or the symbol
O-plus``$\oplus$'' defined by \eqref{eqn:Oplus}.

The following result generalizes Theorem~\ref{thm:HP2T-MHSid21+212}(A).
\begin{thm}\label{thm:MHS21}
Let $r\in \N$ and $\bfs=(\{2\}^{a_1},1,\dots,\{2\}^{a_r},1)$
where $a_1\in\N$ and $a_j\in\N_0$ for all $j\ge 2$.  Then we have
 \begin{equation*}
H^\star_n(\bfs) =-\sum_{\bfp\in \Pi(2a_1+1,\ldots,2a_r+1)} 2^{\ell(\bfp)} \calH_n(\bfp).
\end{equation*}
\end{thm}

\begin{rem}\label{rem:MHS21}
(a). When $r=1$ our Theorem \ref{thm:MHS21} becomes Theorem \ref{thm:HP2T-MHSid21+212}(A).

(b). When $r=2$ we get: for all $n\in \N$ and $a, b\in \N_0$
\begin{align*}
H^\star_n& (\{2\}^a,1,\{2\}^b,1)=2\sum_{k=1}^n\frac{\binom{n}{k}}{k^{2(a+b)+2}\binom{n+k}{k}}
+4\sum_{k=1}^n\frac{H_{k-1}(2b+1)\binom{n}{k}}{k^{2a+1}\binom{n+k}{k}}.
\end{align*}
When  $r=3$ we have: for all $n\in \N$ and $a, b, c\in \N_0$
\begin{align*}
H^\star_n& (\{2\}^a,1,\{2\}^b,1,\{2\}^c,1)=2\sum_{k=1}^n\frac{\binom{n}{k}}{k^{2(a+b+c)+3}\binom{n+k}{k}}
+4\sum_{k=1}^n\frac{H_{k-1}(2c+1)\binom{n}{k}}{k^{2a+2b+2}\binom{n+k}{k}}\\
& +4\sum_{k=1}^n\frac{H_{k-1}(2b+2c+2)\binom{n}{k}}{k^{2a+1}\binom{n+k}{k}}
+ 8\sum_{k=1}^n\frac{H_{k-1}(2b+1,2c+1)\binom{n}{k}}{k^{2a+1}\binom{n+k}{k}}.
\end{align*}
Using Maple we have verified both formulas numerically for $a,b,c\le 5$ and $n\le 100$.
\end{rem}

We now generalize Theorem~\ref{thm:HP2T-MHSid21+212}(B).
\begin{thm} \label{thm:MHS212}
Suppose  $r\in\N_0$ and $\bfs=(\{2\}^{a_1},1,\dots,\{2\}^{a_r},1,\{2\}^{a_{r+1}})$
where $a_j\in\N_0$ for all $j\le r$ and $a_{r+1}\in\N$. Then we have
\begin{equation*}
H^\star_n(\bfs) = - \sum_{\bfp\in \Pi(2a_1+1,\ldots,2a_r+1,\overline{2a_{r+1}})}  2^{\ell(\bfp)} \calH_n(\bfp).
\end{equation*}
\end{thm}

\begin{rem}\label{rem:MHS2-1formula1}
When $r=0$ Theorem \ref{thm:MHS212} implies \cite[(19)]{HessamiPilehrood2Ta2013}.
When $r=1$ Theorem \ref{thm:MHS212} becomes Theorem \ref{thm:HP2T-MHSid21+212}(B).
When $r=2$ we get the following: for all $n,c\in \N$ and $a,b\in \N_0$
\begin{multline}\label{equ:MHS2-1formular=2}
H^\star_n(\{2\}^a,1,\{2\}^b,1,\{2\}^c)
=-2\sum_{k=1}^n\frac{(-1)^{k}\binom{n}{k}}{k^{2(a+b+c)+2}\binom{n+k}{k}}
-4\sum_{k=1}^n\frac{H_{k-1}(\ol{2c})\binom{n}{k}}{k^{2a+2b+2}\binom{n+k}{k}}\\
-4\sum_{k=1}^n\frac{H_{k-1}(\ol{2b+1+2c})\binom{n}{k}}{k^{2a+1}\binom{n+k}{k}}
-8\sum_{k=1}^n\frac{H_{k-1}(2b+1,\ol{2c})\binom{n}{k}}{k^{2a+1}\binom{n+k}{k}}.
\end{multline}
\end{rem}

\section{A combinatorial lemma}
In this short section we prove the following combinatorial identities which are similar in spirit to
\cite[Lemma~2.2]{HessamiPilehrood2Ta2013}. We will need these results several times throughout this paper.
\begin{lem}\label{lem:PTlemma2.2}
Let $k, n\in\N$, $a\in\N_0$,
$A^{(m)}_{n,k}=(-1)^{k}\binom{mn}{n-k}c_n^{(m)}$ where $c_n^{(m)}$ is an arbitrary sequence independent of $k$,
and $B^{(m)}_{n,k}=\binom{mn}{n-k}c_n^{(m)}$. Suppose $a\in \N_0$ and $c\in\N$. Then

\begin{enumerate}
\item[\upshape (i)]
We have
\begin{align}\label{lem:combinatorial}
\frac{1}{n^c}\sum_{k=1}^n \frac{H_{k-1}(\bfv)A^{(m)}_{n,k} }{k^a}
=\sum_{k=1}^n \frac{H_{k-1}(\bfv)A^{(m)}_{n,k}}{k^{a+c}}
+\underset{j\ge 0, x_r>a}{\sum_{j+|\bfx|=a+c}}
 m^{l(\bfx)} \sum_{k=1}^n\frac{H_{k-1}(\bfx,\bfv)A^{(m)}_{n,k}}{k^{j}},
\end{align}
where $x_r$ denotes the last component of $\bfx\in \N^r$.

\item[\upshape (ii)]
We have
\begin{align}\label{equ:PTlemma2.2b}
n\sum_{k=1}^n \frac{H_{k-1}(\bfv)B^{(2)}_{n,k}}{k^a}
=\sum_{k=1}^n \frac{H_{k-1}(\bfv)B^{(2)}_{n,k}}{k^{a-1}}
+ 2\sum_{k=1}^n H_{k-1}(a,\bfv) k B^{(2)}_{n,k}.
\end{align}

\item[\upshape (iii)]
We have
\begin{align}\label{equ:PTlemma2.2c}
\frac{1}{n^c}\sum_{k=1}^n \frac{H_{k-1}(\bfv)B^{(m)}_{n,k} }{k^a}
=\sum_{k=1}^n \frac{H_{k-1}(\bfv)B^{(m)}_{n,k}}{k^{a+c}}
+\underset{j\ge 0, x_r<-a}{\sum_{j+|\bfx|=a+c}}
 m^{l(\bfx)} \sum_{k=1}^n\frac{H_{k-1}(\bfx,\bfv)A^{(m)}_{n,k}}{k^{j}},
\end{align}
where $x_r$ denotes the last component of $\bfx\in \N^{r-1}\times \Z^*$.

\item[\upshape (iv)]
We have
\begin{align}\label{equ:PTlemma2.2d}
n\sum_{k=1}^n \frac{H_{k-1}(\bfv)A^{(2)}_{n,k} }{k^a}
=\sum_{k=1}^n \frac{H_{k-1}(\bfv)A^{(2)}_{n,k}}{k^{a-1}}
+ 2 \sum_{k=1}^n H_{k-1}(\ol{a},\bfv)k B^{(2)}_{n,k}.
\end{align}
\end{enumerate}
\end{lem}
\begin{proof}
We need to mention again that the ordering is reversed in this paper so $s_1$ in \cite[Lemma~2.2]{HessamiPilehrood2Ta2013}
should be the last component of $\bfx$ in our setup. Now, equation \eqref{lem:combinatorial}
follows from \cite[Lemma~2.2]{HessamiPilehrood2Ta2013} directly. We may also use this proof for
\eqref{equ:PTlemma2.2c} by taking the sign of $x_r$ into consideration.

Now by the identity proved in \cite[Lemma 2.1]{HessamiPilehrood2Ta2013}
\begin{equation}\label{equ:PTlem2.1(7)}
2\sum_{k=l+1}^n\frac{k\binom{n}{k}}{\binom{n+k}{k}}=
\frac{n\binom{n-1}{l}}{\binom{n+l}{l}}=\frac{(n-l)\binom{n}{l}}{\binom{n+l}{l}}
\end{equation}
we see that
\begin{align*}
2\sum_{k=1}^n H_{k-1}(a,\bfv) k B^{(2)}_{n,k}
&=\sum_{l=1}^n \frac{H_{l-1}(\bfv)}{l^{a}}\sum_{k=l+1}^n 2kB^{(2)}_{n,k}\\
&=\sum_{l=1}^n\frac{H_{l-1}(\bfv)}{l^{a}} (n-l) B^{(2)}_{n,l}\\
&=n\sum_{l=1}^n\frac{H_{l-1}(\bfv) B^{(2)}_{n,l}}{l^{a}}-\sum_{l=1}^n\frac{H_{l-1}(\bfv) B^{(2)}_{n,l}}{l^{a-1}}
\end{align*}
which is \eqref{equ:PTlemma2.2b}. Similar argument yields \eqref{equ:PTlemma2.2d}.
We leave the details to the interested reader.
\end{proof}

\begin{rem} \label{rem:choiceAm}
In this paper we will always choose $c_n^{(1)}=1$ so that $A^{(1)}_{n,k}=\binom{n}{k}$ and
 $c_n^{(2)}=(n!)^2/(2n)!$ so that $A^{(2)}_{n,k}=\binom{n}{k}/\binom{n+k}{k}.$
\end{rem}

\section{Proof of Theorems~\ref{thm:MHS21} and \ref{thm:MHS212}, $\oll{\bt\hyf 1}$ and $\oll{\bt\hyf 1}\hyf\bt$ formulas}\label{sec:MZSV21&212proof}
In the case $n=1$ both statements in
Theorems~\ref{thm:MHS21} and \ref{thm:MHS212} become $1=1$ trivially. We now assume $n\ge 2$. By definition,
for $\bfs=(\{2\}^{a_1},1,\dots,\{2\}^{a_r},1)$ we have
\begin{align*}
    H^\star_n(\bfs)=&\sum_{l=0}^{a_1} \frac{1}{n^{2a_1-2l}} H^\star_{n-1}(\{2\}^l,1,\{2\}^{a_2},1,\ldots,\{2\}^{a_r},1)\\
    +&\frac{1}{n^{2a_1+1}}H^\star_n(\{2\}^{a_2},1,\ldots,\{2\}^{a_r},1).
\end{align*}
By induction on $r+n$ we see that
\begin{align*}
    H^\star_n(\bfs)=&\sum_{l=0}^{a_1} \frac{1}{n^{2a_1-2l}}
    \sum_{\bfq\in\Pi(2l+1,2a_2+1,\ldots,2a_r+1)}  2^{\ell(\bfq)}\calH_{n-1}(\bfq) \\
    +& \frac{1}{n^{2a_1+1}}   \sum_{\bfp\in\Pi(2a_2+1,\ldots,2a_r+1)}   2^{\ell(\bfp)} \calH_n(\bfp)  \notag \\
=& \sum_{l=0}^{a_1} \frac{1}{n^{2a_1-2l}}
 \sum_{(q_1,\dots,q_m)\in\Pi(1,2a_2+1,\ldots,2a_r+1)}  2^m \sum_{k=1}^{n-1} \frac{1}{k^{2l+q_1} } \sqfr{n-1}{k}\,
 H_{k-1}(q_2,\dots,q_m)   \\
    +&\frac{1}{n^{2a_1+1}}   \sum_{\bfp\in\Pi(2a_2+1,\ldots,2a_r+1)}   2^{\ell(\bfp)} \calH_n(\bfp)
\end{align*}
By changing the order of summations and using the identity
\begin{equation}\label{equ:geomSum}
\sum_{l=0}^{a_1} \left(\frac{n}{k}\right)^{2l}=\frac{1}{k^{2a_1}}\cdot \frac{n^{2a_1+2}-k^{2a_1+2}}{(n-k)(n+k)}
\end{equation}
we see easily that
\begin{align}
  H^\star_n(\bfs)=&
 \sum_{(q_1,\dots,q_m)\in\Pi(1,2a_2+1,\ldots,2a_r+1)}  \sum_{k=1}^n \Big(1-\frac{k^{2a_1+2}}{n^{2a_1+2}} \Big)
 \frac{2^m }{k^{2a_1+q_1} } \sqfr{n}{k}\,  H_{k-1}(q_2,\dots,q_m)  \notag  \\
    +&\frac{1}{n^{2a_1+1}}   \sum_{\bfp\in\Pi(2a_2+1,\ldots,2a_r+1)}   2^{\ell(\bfp)} \calH_n(\bfp)      \notag \\
 =& \sum_{\bfp\in\Pi(2a_1+1,2a_2+1,\ldots,2a_r+1)}  2^{\ell(\bfp)} \calH_n(\bfp)
 +\sum_{(a, \bfv) \in\Pi(2a_2+1,\ldots,2a_r+1)}   2^{\ell(\bfv)+1} \cdot  \notag  \\
 &  \cdot  \left( n    \sum_{k=1}^n \frac{H_{k-1}(\bfv)\binom{n}{k}}{k^a\binom{n+k}{k}}
 - \sum_{k=1}^{n}\frac{H_{k-1}(\bfv)\binom{n}{k}}{k^{a-1} \binom{n+k}{k}}
 -2\sum_{k=1}^{n} \frac{k H_{k-1}(a,\bfv)\binom{n}{k}} {\binom{n+k}{k}}\right) . \label{equ:vashingTerm}
\end{align}
Here the middle (resp.\ the last) term of \eqref{equ:vashingTerm} is obtained by taking
$(q_1,\dots,q_m)=(1+a,\bfv)$
(resp.\  $(q_1,\dots,q_m)=(1,a,\bfv).$
Notice the expression inside the pair of parentheses above vanishes by
\eqref{equ:PTlemma2.2b} of Lemma~\ref{lem:PTlemma2.2}.
This completes the proof of Theorem~\ref{thm:MHS21}.
Exactly the same argument works almost word for word for
Theorem~\ref{thm:MHS212} so we leave the details to the interested reader.

We can now prove Theorems~\ref{thm:MZSV21} and \ref{thm:MZSV212}
by using Theorems~\ref{thm:MHS21} and \ref{thm:MHS212} and
the following key lemma proved in \cite{LinebargerZh2013}.
\begin{lem} \label{lem:limitBound} \emph{\cite[Lemma~4.2]{LinebargerZh2013}}
Let $d\in\N_0$ and let $e$ be a real number with $e>1$. Then for all $\bfs\in(\Z^*)^d$ ($\bfs=\emptyset$ if $d=0$) we have
 \begin{equation}\label{equ:limit0}
\lim_{n\to\infty} \sum_{k=1}^n\frac{|H_{k-1} (\bfs)|}
 {k^e} \left(1-\frac{\binom{n}{k}}{\binom{n+k}{k}}\right)=0.
\end{equation}
\end{lem}

\noindent\emph{Proof of Theorem}~\ref{thm:MZSV21} \emph{and Theorem}~\ref{thm:MZSV212}.
We observe that in Theorem~\ref{thm:MHS21} the first component
$\ge 2a_1+1\ge 3$, and in Theorem~\ref{thm:MHS212}
the absolute value of first component $\ge 2a_1+1\ge 3$.
Therefore both theorems follow from Lemma~\ref{lem:limitBound} immediately.
\hfill$\square$

\medskip

\begin{rem} \label{rem:specialCase}
In \cite{Yamamoto2012b} Yamamoto considers some algebraic structures depending on a variable $t$
which reflect the properties of MZV and MZSV when $t=0$ and $t=1$, respectively. As he pointed out \cite[Conjecture~4.4]{Yamamoto2012a}
the validity of the $\oll{\bt\hyf 1}$ formula Theorem~\ref{thm:MZSV21} implies that the the algebra structure
of MZSVs of the form $\zeta^\star(\{2\}^{a_1},1,\dots,\{2\}^{a_r},1)$ is reflected by setting $t=1/2$.
\end{rem}

\section{MHS and MZSV identities: $\oll{\bt\hyf c\hyf \bt\hyf1}$ formula or $\bt\hyf1\hyf\oll{\bt\hyf c\hyf\bt\hyf1}$ formula}\label{sec:MHSMZSV2c21&212c21}
In this and the next three sections we utilize the ideas in the previous sections to derive more MHS identities
involving arguments of $(\{2\}^a,1)$-type alternating with those of $(\{2\}^b,c)$-type ($c\ge 3$).
In this section, we start by considering strings ending with $(\{2\}^a,1)$-type.
As before, we study the MHS first and then derive the corresponding MZSV identities by invoking Lemma~\ref{lem:limitBound}.

\begin{thm}\label{thm:MHS2c21&212c21}
Let $a_j,b_j,c_j-3\in\N_0$ for all $j\ge 0$. Then

$(\oll{\bt\hyf c\hyf\bt\hyf1})$: For
$\bfs=(\{2\}^{b_1},c_1,\{2\}^{a_1},1,\dots,\{2\}^{b_r},c_r,\{2\}^{a_r},1),$ $r\ge 1$, we have
\begin{equation*}
 H_n^{\star}(\bfs)=\sum_{\bfp\in \Pi(\ol{2b_1+2},\, \{1\}^{c_1-3},\,  \ol{2a_1+2},\,  \ldots,\,
                            \ol{2b_r+2},\, \{1\}^{c_r-3},\,  \ol{2a_r+2})} 2^{\ell(\bfp)} \calH_n(\bfp).
\end{equation*}

$(\bt\hyf1\hyf\oll{\bt\hyf c\hyf\bt\hyf1})$: For
$\bfs=(\{2\}^{a_0},1,\{2\}^{b_1},c_1,\{2\}^{a_1},1,\ldots,\{2\}^{b_r},c_r,\{2\}^{a_{r}},1)$, $r\ge 0$, we have
\begin{equation*}
 H_n^{\star}(\bfs)=\sum_{\bfp\in \Pi(2a_0+1,\,  \ol{2b_1+2},\{1\}^{c_1-3},\, \ol{2a_1+2},\,  \ldots,\,
                            \ol{2b_r+2},\, \{1\}^{c_r-3},\,  \ol{2a_r+2})}  2^{\ell(\bfp)} \calH_n(\bfp).
\end{equation*}
\end{thm}
\begin{proof}
We proceed by induction on $n$. When $n=1$ the theorem is clear. Assume now the theorem is
true for all $n+r\le N$ where $N\ge 2$. Suppose we have $n\ge 2$ and $n+r=N+1$.
Let $\bfs=(\{2\}^{b_1},c_1,\{2\}^{a_1},1,\dots,\{2\}^{b_r},c_r,\{2\}^{a_r},1)$.
We have by definition
\begin{align*}
    H^\star_n(\bfs)=&\sum_{l=0}^{b_1} \frac{1}{n^{2b_1-2l}}H^\star_{n-1}(\{2\}^l,c_1,\{2\}^{a_1},1,\{2\}^{b_2},\ldots,\{2\}^{b_r},c_r,\{2\}^{a_r},1) \\
    +&\frac{1}{n^{2b_1+c_1}}H^\star_n(\{2\}^{a_1},1,\{2\}^{b_2},c_2,\{2\}^{a_2},1,\ldots,\{2\}^{b_r},c_r,\{2\}^{a_r},1).
\end{align*}
For ease of reading we define the following index sets: for any composition $\bfv$ of integers
\begin{align*}
    I(\bfv)= &J(\bfv,\, \{1\}^{c_1-3},\,  \ol{2a_1+2}),\\
    J(\bfv)=&\Pi(\bfv,\,  \ol{2b_2+2},\{1\}^{c_2-3},\, \ol{2a_2+2},\,  \ldots,\,
                            \ol{2b_r+2},\, \{1\}^{c_r-3},\,  \ol{2a_r+2}).
\end{align*}
Then by induction assumption
\begin{align*}
 H^\star_n(\bfs)=&\sum_{l=0}^{b_1} \frac{1}{n^{2b_1-2l}}
   \sum_{\bfq\in  I(\ol{2l+2})}  2^{\ell(\bfq)} \calH_{n-1}(\bfq)
+\frac{1}{n^{2b_1+c_1}}\sum_{\bfp\in J(2a_1+1) } \calH_n(\bfp) \notag\\
=&  \sum_{(q_1,\dots,q_m)\in I(\ol{2})}
 \sum_{l=0}^{b_1} \frac{2^m}{n^{2b_1-2l}}  \sum_{k=1}^{n-1} \frac{\sgn(q_1)^k}{k^{2l+|q_1|} } \sqfr{n-1}{k}\,
 H_{k-1}(q_2,\dots,q_m) \\
+&\frac{1}{n^{2b_1+c_1}}\sum_{\bfp\in J(2a_1+1)}  2^{\ell(\bfp)} \calH_n(\bfp), \notag\\
\end{align*}

By changing the order of summations and using the identity \eqref{equ:geomSum}
we see easily that
{\allowdisplaybreaks
\begin{align*}
 H^\star_n(\bfs)=&
    \sum_{(q_1,\dots,q_m)\in I(\ol{2b_1+2})}
    2^m \sum_{k=1}^n \frac{\sgn(q_1)^k}{k^{|q_1|}} \Big(1-\frac{k^{2b_1+2}}{n^{2b_1+2}} \Big)\sqfr{n}{k}\,
 H_{k-1}(q_2,\dots,q_m) \\
+&\frac{1}{n^{2b_1+c_1}}\sum_{\bfp\in J(2a_1+1)}  2^{\ell(\bfp)} \calH_n(\bfp). \notag\\
\end{align*}
}
Notice that the index set
\begin{equation} \label{equ:myIndexSet}
 I(\ol{2b_1+2})=\bigcup_{(p_1,\dots,p_m)\in J(2a_1+1)} \Big\{\big( \ol{2b_1+2} \circ \underbrace{1 \circ \dots \circ 1}_{c_1-3 \text{ times}}
\circ (\bar 1 \oplus p_1),\dots,p_m \big)  \Big\}.
\end{equation}
For each $(p_1,\dots,p_m)$ we can partition the set \eqref{equ:myIndexSet} into the following subsets:
\begin{alignat*}{4}
 \Big\{ \big((2b_1+c_1)\oplus p_1,\bfv \big)\Big\} \cup
 \Big\{\big(\ol{j+2+2b_1},\bfx,\bar{i}\oplus p_1,\bfv\big)\Big\}, \quad
 \bfv=(p_2,\dots,p_m),
\end{alignat*}
for $i\ge 1$, $j\ge 0$ and positive compositions $\bfy$ with $i+j+|\bfy|=c_1-2$.
We have two cases: (i) $p_1>0$ and (ii) $p_1<0$.
Set $a=|p_1|$. It suffices to prove that in case (i)
\begin{multline}\label{equ:usePTLemma2.2a}
2^{l(\bfy)+1} \sum_{\substack{\ i+j+|\bfy|=c_1-2, \\ i\ge  1,j\ge 0}}
\sum_{k=1}^n\frac{H_{k-1}(\bfy,\ol{i+a},\bfv) (-1)^{k}\binom{n}{k}} {k^{j} \binom{n+k}{k}}\\
=\frac{1}{n^{c_1-2}} \sum_{k=1}^n\frac{H_{k-1}(\bfv) \binom{n}{k}} {k^{a} \binom{n+k}{k}}
-\sum_{k=1}^n\frac{H_{k-1}(\bfv) \binom{n}{k}} {k^{c_1+a-2} \binom{n+k}{k}} ,
\end{multline}
and in case (ii)
\begin{multline}\label{equ:usePTLemma2.2c}
2^{l(\bfy)+1} \sum_{\substack{\ i+j+|\bfy|=c_1-2, \\ i\ge 1,j\ge 0}}
\sum_{k=1}^n\frac{H_{k-1}(\bfy,i+a,\bfv) (-1)^{k}\binom{n}{k}} {k^{j} \binom{n+k}{k}}\\
=\frac{1}{n^{c_1-2}} \sum_{k=1}^n\frac{H_{k-1}(\bfv)  (-1)^{k}\binom{n}{k}} {k^{a} \binom{n+k}{k}}
-\sum_{k=1}^n\frac{H_{k-1}(\bfv)  (-1)^{k}\binom{n}{k}} {k^{c_1+a-2} \binom{n+k}{k}}.
\end{multline}
Equation \eqref{equ:usePTLemma2.2a} follows from \eqref{equ:PTlemma2.2c}
of Lemma~\ref{lem:PTlemma2.2} when
$c=c_1-2$,$ m=2$ and $\bfx=(\bfy,\ol{i+a})$.
Equation \eqref{equ:usePTLemma2.2c} follows from \eqref{lem:combinatorial}  with
the same choice of parameters except $\bfx=(\bfy,i+a)$.
This proves case {\rm ($\oll{\bt\hyf c\hyf\bt\hyf1}$)} when $r+n=N+1$.

Now we turn to case {\rm ($\bt\hyf1\hyf\oll{\bt\hyf c\hyf\bt\hyf1}$)}. Let
$\bfs=(\{2\}^{a_0},1,\{2\}^{b_1},c_1,\{2\}^{a_1},1,\dots,\{2\}^{b_r},c_r,$ $\{2\}^{a_r},1).$
By definition
\begin{align*}
    H^\star_n(\bfs)=& \sum_{l=0}^{a_0} \frac{1}{n^{2a_0-2l}}H^\star_{n-1}(\{2\}^l,1,\{2\}^{b_1},c_1,\{2\}^{a_1},1,\dots,\{2\}^{b_r},c_r,\{2\}^{a_r},1) \\
    +&\frac{1}{n^{2a_0+1}}H^\star_n(\{2\}^{b_1},c_1,\{2\}^{a_1},1,\dots,\{2\}^{b_r},c_r,\{2\}^{a_r},1).
\end{align*}
For ease of reading, for any composition $\bfv$ of integers we set
$$K(\bfv)=\Pi(\bfv,\,  \ol{2b_1+2},\{1\}^{c_1-3},\, \ol{2a_1+2},\,  \ldots,\,
                            \ol{2b_r+2},\, \{1\}^{c_r-3},\,  \ol{2a_r+2}) .$$
By the induction assumption and the fact that we have just proved the case
{\rm ($\oll{\bt\hyf c\hyf\bt\hyf1}$)} with $r+n=N+1$ we see that
{\allowdisplaybreaks
\begin{align*}
    H^\star_n(\bfs)=&\sum_{l=0}^{a_0} \frac{1}{n^{2a_0-2l}}
    \sum_{\bfq\in K(2l+1)}  2^{\ell(\bfq)} \calH_{n-1}(\bfq)
    +\frac{1}{n^{2a_0+1}} \sum_{\bfp\in K(\emptyset)} 2^{\ell(\bfp)} \calH_n(\bfp)  \\
 =& \sum_{(q_1,\dots,q_m)\in K(1)}
 \sum_{l=0}^{a_0} \frac{2^m}{n^{2a_0-2l}} \sum_{k=1}^{n-1} \frac{\sgn(q_1)^k}{k^{2l+|q_1|} } \sqfr{n-1}{k}\,
 H_{k-1}(q_2,\dots,q_m) \\
+&\frac{1}{n^{2a_0+1}} \sum_{\bfp\in K(\emptyset)} 2^{\ell(\bfp)} \calH_n(\bfp).
\end{align*}}
By changing the order of summations and using the identity \eqref{equ:geomSum} we get
\begin{align*}
    H^\star_n(\bfs)=&
    \sum_{(q_1,\dots,q_m)\in K(2a_0+1)}  2^m \sum_{k=1}^n \Big(1-\frac{k^{2a_0+2}}{n^{2a_0+2}}\Big)
    \frac{\sgn(q_1)^k}{k^{|q_1|} } \sqfr{n}{k}\, H_{k-1}(q_2,\dots,q_m) \\
+&\frac{1}{n^{2a_0+1}} \sum_{\bfp\in K(\emptyset) } 2^{\ell(\bfp)} \calH_n(\bfp).
\end{align*}
Notice that the partition of the index set
\begin{align*}
K(2a_0+1)  =\bigcup_{(p_1,\dots,p_m)\in K(\emptyset) } \Big\{\big(2a_0+1,p_1,\bfv \big)  \Big\} \cup
\Big\{\big( (2a_0+1)\oplus p_1,\bfv \big)  \Big\},
\end{align*}
where $\bfv=(p_2,\dots,p_m).$
For each $(p_1,\dots,p_m)$ we have two cases: (i) $p_1>0$ and (ii) $p_1<0$.
Set $a=|p_1|$. It suffices to prove that in case (i)
\begin{equation*}
n\sum_{k=1}^n\frac{H_{k-1}(\bfv) \binom{n}{k}} {k^a \binom{n+k}{k}}
-2  \sum_{k=1}^n\frac{kH_{k-1}(a,\bfv)\binom{n}{k}} {\binom{n+k}{k}}
-\sum_{k=1}^n\frac{H_{k-1}(\bfv) \binom{n}{k}} {k^{a-1} \binom{n+k}{k}}=0
\end{equation*}
by \eqref{equ:PTlemma2.2b} of Lemma~\ref{lem:PTlemma2.2}, and in case (ii)
\begin{equation*}
n\sum_{k=1}^n\frac{H_{k-1}(\bfv)(-1)^k\binom{n}{k}} {k^a \binom{n+k}{k}}
-2  \sum_{k=1}^n\frac{kH_{k-1}(\ol{a},\bfv)\binom{n}{k}} {\binom{n+k}{k}}
-\sum_{k=1}^n\frac{H_{k-1}(\bfv)(-1)^k\binom{n}{k}} {k^{a-1} \binom{n+k}{k}}=0
\end{equation*}
by \eqref{equ:PTlemma2.2d} of Lemma~\ref{lem:PTlemma2.2}.
This proves case {\rm ($\bt\hyf1\hyf\oll{\bt\hyf c\hyf\bt\hyf1}$)} when $r+n=N+1$.

We have completed the proof of Theorem~\ref{thm:MHS2c21&212c21}.
\end{proof}

When $r=1$ we get the following: for all $n\in \N$ and $a,b\in \N_0$
\begin{equation}
\label{equ:MHS2321r=1}
H^\star_n(\{2\}^b,3,\{2\}^a,1)
=2\sum_{k=1}^n\frac{\binom{n}{k}}{k^{2(b+a)+4}\binom{n+k}{k}}
+4\sum_{k=1}^n\frac{(-1)^k H_{k-1}(\ol{2a+2})\binom{n}{k}}{k^{2b+2}\binom{n+k}{k}},
\end{equation}
in case {\rm ($\oll{\bt\hyf c\hyf\bt\hyf1}$)}, and in case {\rm ($\bt\hyf1\hyf\oll{\bt\hyf c\hyf\bt\hyf1}$)}
\begin{align}\label{equ:MHS212321r=1}
H^\star_n(\{2\}^{a_1},1,& \{2\}^{b},3,\{2\}^{a_2},1) \\
=&2\sum_{k=1}^n\frac{\binom{n}{k}}{k^{2(a_1+b+a_2+5)}\binom{n+k}{k}}
+4\sum_{k=1}^n\frac{H_{k-1}(2b+2a_2+4)\binom{n}{k}}{k^{2a_1+1}\binom{n+k}{k}}  \notag \\
+&4\sum_{k=1}^n\frac{(-1)^k H_{k-1}(\ol{2a_2+2})\binom{n}{k}}{k^{2a_1+2b+3}\binom{n+k}{k}}
+8\sum_{k=1}^n\frac{H_{k-1}(\ol{2b+2},\ol{2a_2+2})\binom{n}{k}}{k^{2a_1+1}\binom{n+k}{k}} . \notag
\end{align}

\begin{thm}\label{thm:MZSV2c21&212c21}
Let $r\in\N$ and $a_j,b_j,c_j-3\in\N_0$ for all $j\ge 1$. Then

$(\oll{\bt\hyf c\hyf\bt\hyf1})$: For
$\bfs=(\{2\}^{b_1},c_1,\{2\}^{a_1},1,\dots,\{2\}^{b_r},c_r,\{2\}^{a_r},1),$ $r\ge 1$, we have
\begin{equation}\label{equ:MZSV2c21}
\zeta^{\star}(\bfs)=\sum_{\bfp\in \Pi(\ol{2b_1+2},\, \{1\}^{c_1-3},\,  \ol{2a_1+2},\,  \ldots,\,
                            \ol{2b_r+2},\, \{1\}^{c_r-3},\,  \ol{2a_r+2})} 2^{\ell(\bfp)} \zeta(\bfp).
\end{equation}

$(\bt\hyf1\hyf\oll{\bt\hyf c\hyf\bt\hyf1})$: For
$\bfs=(\{2\}^{a_0},1,\{2\}^{b_1},c_1,\{2\}^{a_1},1,\ldots,\{2\}^{b_r},c_r,\{2\}^{a_{r}},1)$, $r\ge 0$
and $a_0\ge 1$, we have
\begin{equation}\label{equ:MZSV212c21}
\zeta^{\star}(\bfs)=\sum_{\bfp\in \Pi(2a_0+1,\,  \ol{2b_1+2},\{1\}^{c_1-3},\, \ol{2a_1+2},\,  \ldots,\,
                            \ol{2b_r+2},\, \{1\}^{c_r-3},\,  \ol{2a_r+2})} 2^{\ell(\bfp)} \zeta(\bfp).
\end{equation}
\end{thm}
\begin{proof}
This follows from Theorem~\ref{thm:MHS2c21&212c21} and Lemma \ref{lem:limitBound} easily.
\end{proof}
For example, by \eqref{equ:MHS2321r=1} and \eqref{equ:MHS212321r=1} we see that
\begin{equation}\label{equ:MZSV2321r=1}
\zeta^\star(\{2\}^b,3,\{2\}^a,1)
=2\zeta(2a+2b+4)+4\zeta(\ol{2b+2},\ol{2a+2}),
\end{equation}
and
\begin{multline*}
\zeta^\star(\{2\}^{a_1},1,\{2\}^{b},3,\{2\}^{a_2},1)
=2\zeta(2(a_1+b+a_2)+5)
+4\zeta(2a_1+1,2b+2a_2+4) \\
+4\zeta(\ol{2a_1+2b+3},\ol{2a_2+2})
+8\zeta(2a_1+1,\ol{2b+2},\ol{2a_2+2}).
\end{multline*}

\section{MHS and MZSV identities: $\oll{\bt\hyf c\hyf \bt\hyf1}\hyf \bt$ and $\bt\hyf1\hyf\oll{\bt\hyf c\hyf\bt\hyf1}\hyf \bt$ formula}\label{sec:MHSMZSV2c212&212c212}
We continue to study MHS identities
involving arguments of $(\{2\}^a,1)$-type alternating with those of $(\{2\}^b,c)$-type ($c\ge 3$).
Now we consider strings ending with $(\{2\}^a,1)$-type trailed by a non-empty substring of 2's at the very end.

\begin{thm}\label{thm:MHS2c212&212c212}
Let $t,r\in\N$ and $a_j,b_j,c_j-3\in\N_0$ for all $j\ge 1$. Then

$(\oll{\bt\hyf c\hyf\bt\hyf1}\hyf\bt)$. For $\bfs=(\{2\}^{b_1},c_1,\{2\}^{a_1},1,\dots,\{2\}^{b_r},c_r,\{2\}^{a_r},1,\{2\}^t)$ we have
\begin{equation*}
 H_n^{\star}(\bfs)=-\sum_{\bfp\in \Pi(\ol{2b_1+2},\, \{1\}^{c_1-3},\,  \ol{2a_1+2},\,  \ldots,\,
                            \ol{2b_r+2},\, \{1\}^{c_r-3},\,  \ol{2a_r+2}, \ol{2t})} 2^{\ell(\bfp)} \calH_n(\bfp).
\end{equation*}

$(\bt\hyf1\hyf\oll{\bt\hyf c\hyf\bt\hyf1}\hyf\bt)$. For $\bfs=(\{2\}^{a_0},1,\{2\}^{b_1},c_1,\{2\}^{a_1},1,\ldots,\{2\}^{b_r},c_r,\{2\}^{a_{r}},1,\{2\}^t)$, $r\ge 0$, we have
\begin{equation*}
 H_n^{\star}(\bfs)=-\sum_{\bfp\in \Pi(2a_0+1,\,  \ol{2b_1+2},\{1\}^{c_1-3},\, \ol{2a_1+2},\,  \ldots,\,
                            \ol{2b_r+2},\, \{1\}^{c_r-3},\,  \ol{2a_r+2}, \ol{2t})}  2^{\ell(\bfp)} \calH_n(\bfp).
\end{equation*}
\end{thm}

\begin{proof}
The proof is very similar to the proof of Theorem~\ref{thm:MHS2c21&212c21}.
Thus we leave it to the interested reader.
\end{proof}
By taking $n\to\infty$ and using Lemma~\ref{lem:limitBound} we get immediately the following results.

\begin{thm}\label{thm:MZSV2c212&212c212}
Let $t,r\in\N$ and $a_j,b_j,c_j-3\in\N_0$ for all $j\ge 1$. Then

$(\oll{\bt\hyf c\hyf\bt\hyf1}\hyf\bt)$. For
$\bfs=(\{2\}^{b_1},c_1,\{2\}^{a_1},1,\dots,\{2\}^{b_r},c_r,\{2\}^{a_r},1,\{2\}^t),$ $r\ge 1$, we have
\begin{equation}\label{equ:MZSV2c212}
\zeta^{\star}(\bfs)=-\sum_{\bfp\in \Pi(\ol{2b_1+2},\, \{1\}^{c_1-3},\,  \ol{2a_1+2},\,  \ldots,\,
                            \ol{2b_r+2},\, \{1\}^{c_r-3},\,  \ol{2a_r+2}, \ol{2t})} 2^{\ell(\bfp)} \zeta(\bfp).
\end{equation}

$(\bt\hyf1\hyf\oll{\bt\hyf c\hyf\bt\hyf1}\hyf\bt)$: For
$\bfs=(\{2\}^{a_0},1,\{2\}^{b_1},c_1,\{2\}^{a_1},1,\ldots,\{2\}^{b_r},c_r,\{2\}^{a_{r}},1,\{2\}^t)$, $r\ge 0$ and $a_0\ge 1$, we have
\begin{equation}\label{equ:MZSV212c212}
\zeta^{\star}(\bfs)=-\sum_{\bfp\in \Pi(2a_0+1,\,  \ol{2b_1+2},\{1\}^{c_1-3},\, \ol{2a_1+2},\,  \ldots,\,
                            \ol{2b_r+2},\, \{1\}^{c_r-3},\,  \ol{2a_r+2}, \ol{2t})} 2^{\ell(\bfp)} \zeta(\bfp).
\end{equation}
\end{thm}

For example, taking $r=1$ and $c_1=3$ we get in case ($\oll{\bt\hyf c\hyf\bt\hyf1}\hyf\bt$)
\begin{align*}
\zeta^\star(\{2\}^b,3,\{2\}^a,1,\{2\}^t)
=-2&\zeta(\ol{2b+2a+2t+4})-4\zeta(2b+2a+4,\ol{2t})\\
-4&\zeta(\ol{2b+2},2a+2t+2)-8\zeta(\ol{2b+2},\ol{2a+2},\ol{2t}),
\end{align*}
and in case ($\bt\hyf1\hyf\oll{\bt\hyf c\hyf\bt\hyf1}\hyf\bt$)
\begin{align*}
&\zeta^\star(\{2\}^{a_1},1,\{2\}^b,3,\{2\}^{a_2},1,\{2\}^t)\\
=&-2\zeta(\ol{2a_1+2b+2a_2+2t+5})-4\zeta(2a_1+2b+2a_2+5,\ol{2t})\\
&-4\zeta(\ol{2a_1+2b+3},2a_2+2t+2)-8\zeta(\ol{2a_1+2b+3},\ol{2a_2+2},\ol{2t})\\
&-4\zeta(2a_1+1,\ol{2b+2a_2+2t+4})-8\zeta(2a_1+1,2b+2a_2+4,\ol{2t})\\
&-8\zeta(2a_1+1,\ol{2b+2},2a_2+2t+2)-16\zeta(2a_1+1,\ol{2b+2},\ol{2a_2+2},\ol{2t}).
\end{align*}
We have verified these formulas numerically for $a_1,a_2, a,b,t\le 2$ using EZ-face \cite{EZface}.

\section{MHS identities: $\oll{\bt\hyf1\hyf\bt\hyf c}\hyf\bt$ formula and $\bt\hyf c\ol{\hyf \bt\hyf1\hyf\bt\hyf c}\hyf\bt$ formula}\label{sec:MHSMZSV212c2&2c212c2}
In this section we continue the theme in the preceding sections by deriving more MHS identities
involving compositions of $(\{2\}^a,1)$-type alternating with those of $(\{2\}^b,c)$-type ($c\ge 3$).
This time we consider strings ending with $(\{2\}^a,c)$-type trailed by $\{2\}^t$ (this tail may be empty).
We can then derive the corresponding MZSV identities by applying Lemma~\ref{lem:limitBound}. We omit the
proofs since they are essentially the same as those of Theorem~\ref{thm:MHS2c21&212c21} and
Theorem~\ref{thm:MZSV2c21&212c21}.

\begin{thm}\label{thm:MHS212c2&2c212c2}
Let $r\in\N$ and $t,a_j,b_j,c_j-3\in\N_0$ for all $j\ge 1$. Then

$(\oll{\bt\hyf 1\hyf\bt\hyf c}\hyf\bt)$. For $\bfs=(\{2\}^{a_1},1,\{2\}^{b_1},c_1,\dots,\{2\}^{a_r},1,\{2\}^{b_r},c_r,\{2\}^t)$, we have
\begin{equation*}
 H_n^{\star}(\bfs)=-\sum_{\bfp\in \Pi(2a_1+1,\, \ol{2b_1+2},\, \{1\}^{c_1-3},\,  \ol{2a_2+2},\,  \ldots,\,
                             \ol{2a_r+2},\,  \ol{2b_r+2},\, \{1\}^{c_r-3},\, 2t+1)} 2^{\ell(\bfp)} \calH_n(\bfp).
\end{equation*}

$(\bt\hyf c\hyf\oll{\bt\hyf 1\hyf\bt\hyf c}\hyf\bt)$. For $\bfs=(\{2\}^{b_1},c_1,\{2\}^{a_1},1,\dots,\{2\}^{b_r},c_r,\{2\}^{a_r},1,\{2\}^{b_{r+1}},c_{r+1},\{2\}^t)$ we have
\begin{equation*}
 H_n^{\star}(\bfs)=-\sum_{\bfp\in \Pi( \ol{2b_1+2},\{1\}^{c_1-3},\, \ol{2a_1+2},\,  \ldots,\,\ol{2a_r+2},\,
                            \ol{2b_{r+1}+2},\, \{1\}^{c_{r+1}-3},\,  2t+1)}  2^{\ell(\bfp)} \calH_n(\bfp).
\end{equation*}
\end{thm}

Setting $r=0$ in Theorem~\ref{thm:MHS212c2&2c212c2} we recover \cite[Theorem~2.1]{HessamiPilehrood2Ta2013}.
When $r=1$ and $t=0$ we get the following: for all $n\in \N$ and $a,b\in \N_0$
\begin{align*}
H^\star_n(\{2\}^a,1,\{2\}^b,3)
=& -2\sum_{k=1}^n\frac{(-1)^{k}\binom{n}{k}}{k^{2(a+b)+4}\binom{n+k}{k}}
- 4\sum_{k=1}^n\frac{H_{k-1}(\ol{2b+3})\binom{n}{k}} {k^{2a+1}\binom{n+k}{k}} \\
-& 4 \sum_{k=1}^n\frac{H_{k-1}(1)(-1)^{k}\binom{n}{k}}{k^{2a+2b+3}\binom{n+k}{k}}
-  8 \sum_{k=1}^n\frac{H_{k-1}(\ol{2b+2},1)\binom{n}{k}} {k^{2a+1}\binom{n+k}{k}}
\end{align*}
in case {\rm ($\oll{\bt\hyf1\hyf\bt\hyf c}\hyf\bt$)}, and in case {\rm ($\bt\hyf c\hyf\oll{\bt\hyf1\hyf\bt\hyf c}\hyf\bt$)}:
\begin{align*}
H^\star_n& (\{2\}^{b_1},3,\{2\}^a,1,\{2\}^{b_2},3)
=-2\sum_{k=1}^n\frac{(-1)^{k}\binom{n}{k}}{k^{2(b_1+a+b_2)+7}\binom{n+k}{k}}
 -4\sum_{k=1}^n\frac{H_{k-1}(\ol{2b_2+3})\binom{n}{k}}{k^{2b_1+2a+4}\binom{n+k}{k}}\\
       -&4\sum_{k=1}^n\frac{H_{k-1}(2a+2b_2+5)(-1)^k\binom{n}{k}}{k^{2b_1+2}\binom{n+k}{k}}
      -8\sum_{k=1}^n\frac{H_{k-1}(\ol{2a+2},\ol{2b_2+3})(-1)^k\binom{n}{k}}{k^{2b_1+2}\binom{n+k}{k}}\\
        -&4\sum_{k=1}^n\frac{H_{k-1}(1)(-1)^k\binom{n}{k}}{k^{2b_1+2a+2b_2+6}\binom{n+k}{k}}
      -8\sum_{k=1}^n\frac{H_{k-1}(2a+2b_2+4,1)(-1)^k\binom{n}{k}}{k^{2b_1+2}\binom{n+k}{k}}\\
      -&8\sum_{k=1}^n\frac{H_{k-1}(\ol{2b_2+2},1)\binom{n}{k}}{k^{2b_1+2a+4}\binom{n+k}{k}}
       -16\sum_{k=1}^n\frac{H_{k-1}(\ol{2a+2},\ol{2b_2+2},1)(-1)^k\binom{n}{k}}{k^{2b_1+2}\binom{n+k}{k}}.
\end{align*}

By taking $n\to \infty$ in Theorem \ref{thm:MHS212c2&2c212c2} and using Lemma~\ref{lem:limitBound} we obtain
\begin{thm}\label{thm:MZSV212c2&2c212c2}
Let $r\in\N$ and $a_j,b_j,c_j-3\in\N_0$ for all $j\ge 1$. Then

$(\oll{\bt\hyf 1\hyf\bt\hyf c}\hyf\bt)$. For $\bfs=(\{2\}^{a_1},1,\{2\}^{b_1},c_1,\dots,\{2\}^{a_r},1,\{2\}^{b_r},c_r,\{2\}^t)$ with $a_1\ge 1$, we have
\begin{equation}\label{equ:MZSV212c2}
\zeta^{\star}(\bfs)=-\sum_{\bfp\in \Pi(2a_1+1,\, \ol{2b_1+2},\, \{1\}^{c_1-3},\,  \ol{2a_2+2},\,  \ldots,\,
                             \ol{2a_r+2},\,  \ol{2b_r+2},\, \{1\}^{c_r-3},\, 2t+1)} 2^{\ell(\bfp)} \zeta(\bfp).
\end{equation}

$(\bt\hyf c\hyf\oll{\bt\hyf 1\hyf\bt\hyf c}\hyf\bt)$. For $\bfs=(\{2\}^{b_1},c_1,\{2\}^{a_1},1,\dots,\{2\}^{b_r},c_r,\{2\}^{a_r},1,\{2\}^{b_{r+1}},c_{r+1},\{2\}^t)$ we have
\begin{equation}\label{equ:MZSV2c212c2}
\zeta^{\star}(\bfs)=-\sum_{\bfp\in \Pi( \ol{2b_1+2},\{1\}^{c_1-3},\, \ol{2a_1+2},\,  \ldots,\,\ol{2a_r+2},\,
                            \ol{2b_{r+1}+2},\, \{1\}^{c_{r+1}-3},\,  2t+1)} 2^{\ell(\bfp)} \zeta(\bfp).
\end{equation}
\end{thm}

For example, taking $r=1$ and $t=0$ in case {\rm ($\oll{\bt\hyf 1\hyf\bt\hyf c}\hyf\bt$)} we get
\begin{align*}
\zeta^\star(\{2\}^a,1,\{2\}^b,3)
=& -2\zeta(\ol{2a+2b+4})
- 4\zeta(1+2a,\ol{2b+3})\\
-& 4 \zeta(\ol{2a+2b+3},1)
-  8 \zeta(2a+1,\ol{2b+2},1).
\end{align*}
and in case {\rm ($\bt\hyf c\hyf\oll{\bt\hyf1\hyf\bt\hyf c}\hyf\bt$)} we get
\begin{align*}
& \zeta^\star(\{2\}^{b_1},3,\{2\}^a,1,\{2\}^{b_2},3) \\
=&-2\zeta(\ol{2(b_1+a+b_2)+7})
 -4\zeta(2b_1+2a+4,\ol{2b_2+3}) \\
 &-4\zeta(\ol{2b_1+2},2a+2b_2+5)
-4\zeta(\ol{2b_1+2a+2b_2+6},1) \\
 &-8\zeta(\ol{2b_1+2},\ol{2a+2},\ol{2b_2+3}) -8\zeta(\ol{2b_1+2}, 2a+2b_2+4,1)  \\
 &-8\zeta(2b_1+2a+4,\ol{2b_2+2},1)  -16\zeta(\ol{2b_1+2},\ol{2a+2},\ol{2b_2+2},1).
\end{align*}
We have numerically verified these formulas with $a,b_1,b_2\le 2$ using EZ-face \cite{EZface}.

\section{Conjectures of Imatomi et al. on MZSV of type \\
$\oll{\bt\hyf 3\hyf \bt\hyf1}$ and $\oll{\bt\hyf 3\hyf\bt\hyf 1}\hyf\bt\hyf1$}\label{sec:ITTWConj}
Throughout this section the notations of \S\ref{sec:MHSMZSV2c21&212c21} and \S\ref{sec:MHSMZSV2c212&212c212}
are still in force. The following Theorem~\ref{thm:ITTWConj} was first conjectured by Imatomi et al. \cite[Conjectures 4.1 and 4.3]{ITTW2009}. Special cases have been proved in \cite[Theorem 1.1]{ITTW2009} and by Tasaka and Yamamoto in \cite{TasakaYa2012}. Yamamoto proves a more precise version in \cite{Yamamoto2012a}. We now give a different and concise proof using the identities we have found in the last two sections. We begin with a lemma first.

\begin{lem}\label{lem:evenArg}
Let $n_1,\dots,n_\ell\in \Z^*$ be $\ell$ even integers and $m=|n_1|+\cdots+|n_\ell|$. Then
\begin{multline*}
\sum_{g\in\frakS_{\ell}}\zeta \big(n_{g(1)},\dots,n_{g(\ell)}\big)= \\
\sum_{ e_1+ \cdots + e_p =\ell} (-1)^{\ell-p} \prod_{s=1}^p (e_s-1)!
\sum  \zeta\left(\bigoplus_{k \in \pi_1} n_k\right) \dots \zeta\left(\bigoplus_{k \in \pi_p} n_k\right)\in \Q\pi^{m}
\end{multline*}
where the sum in the right is taken over all the possible unordered partitions
of the set $\{1, \dots, \ell\}$ into $p$ subsets $\pi_1, \dots, \pi_p$ with $e_1, \dots, e_p$ elements respectively.
\end{lem}
\begin{proof}
When all the arguments $n_1,\dots,n_\ell$ are positive the lemma becomes \cite[Theorem 2.2]{Hoffman1992}. Its proof there can be
used here almost word for word. Notice that \cite[Theorem 2.2]{Hoffman1992} is re-proved as \cite[Proposition~9.4]{KurokawaLaOc2009}
whose proof is different from that of \cite{Hoffman1992} but also works here. Thus we leave the details to the interested reader.
\end{proof}

\begin{thm}  \label{thm:ITTWConj}
Let $r$ be a positive integer, and
$e_1,\ldots,e_{2r+1}$ nonnegative integers.
\begin{enumerate}
\item[{\rm (i)}]
Put $m=e_1+\cdots+e_{2r}$. Then we have
\begin{equation*}
\sum_{\tau\in\frakS_{2r}}
\zeta^\star(\{2\}^{e_{\tau(1)}},3,\{2\}^{e_{\tau(2)}},1,\{2\}^{e_{\tau(3)}},
    \ldots,3,\{2\}^{e_{\tau(2r)}},1)
\in \Q\cdot\pi^{2m+4r}.
\end{equation*}

\item[{\rm (ii)}]
 Put $m=e_1+\cdots+e_{2r+1}$. Then we have
\begin{equation*}
\sum_{\tau\in\frakS_{2r+1}}
\zeta^\star(\{2\}^{e_{\tau(1)}},3,\{2\}^{e_{\tau(2)}},1,
    \ldots,3,\{2\}^{e_{\tau(2n)}},1,\{2\}^{e_{\tau(2r+1)}+1})
\in \Q\cdot\pi^{2m+4r+2}.
\end{equation*}
\end{enumerate}
\end{thm}
\begin{proof} We start with (i) first.
When $r=1$ this follows quickly from \eqref{equ:MZSV2321r=1} by stuffle relation. For general $r$
let $a_j=e_{2j}$ and $b_j=e_{2j-1}$ for all $j\le r$ and let $A_j=2e_{j}+2$ for all $j\le 2r$ .
Then we can apply \eqref{equ:MZSV2c21} of Theorem \ref{thm:MZSV2c21&212c21} to the string $\bfs=(\{2\}^{b_1},3,\{2\}^{a_1},1,\dots,\{2\}^{b_r},3,\{2\}^{a_r},1)$ and get
\begin{equation*}
\zeta^\star(\bfs) =
\sum_{\tau\in\frakS_{2r}} \tau\left\{\sum_{\bfp\in\Pi(\ol{A_1},\ldots,\ol{A_{2r}}) }  2^{\ell(\bfp)}\zeta(\bfp) \right\}.
\end{equation*}
Let $\bfA=(\ol{A_1}, \ldots,\ol{A_{2r}})$ and $P_\ell(2r)$ be
the set of all partitions of $[2r]:=\{1,2,\dots,2r\}$ into $\ell$ consecutive subsets.
If $\gl=(\gl_1,\dots,\gl_\ell)\in P_\ell(2r)$ then we set
$\gl_j(\bfA)=(\ol{A_i})_{i\in \gl_j}$ so that the concatenation $\bigsqcup_{j=1}^\ell  \gl_j(\bfA) =\bfA$.
Because of the permutation we see that
\begin{align*}
\zeta^\star(\bfs) =&
 \sum_{\tau\in\frakS_{2r}} \tau \left\{\sum_{\ell=1}^{2r} 2^{\ell}
\sum_{\gl\in P_\ell(2r)}  \zeta\big(\oplus\gl_1(\bfA),\dots,\oplus\gl_\ell(\bfA)\big) \right\}\\
=& \sum_{\tau\in\frakS_{2r}} \tau \left\{\sum_{\ell=1}^{2r} \frac{2^{\ell}}{\ell!}
\sum_{\gl\in P_\ell(2r)} \sum_{g\in \frakS_\ell}
\zeta\big(\oplus\gl_{g(1)}(\bfA),\dots,\oplus\gl_{g(\ell)}(\bfA) \big) \right\},
\end{align*}
where $\oplus\bft$ is the $\oplus$-sum of all the components of $\bft$ for any composition $\bft$.
Hence Theorem~\ref{thm:ITTWConj}(i) follows
readily from the Lemma~\ref{lem:evenArg} since all $A_j$'s are even numbers.

Theorem~\ref{thm:ITTWConj} (ii) follows from Theorem \ref{thm:MZSV2c212&212c212}
in a similar fashion so we leave the details to the interested reader.
\end{proof}

\begin{rem}
We notice that in \cite[Theorem 1.1]{Yamamoto2012a} Yamamoto obtains a more precise formula by using partial sums and generating functions:
\begin{multline}\label{equ:YamamotoThm}
 \sum_{\substack{e_0,e_1,\ldots,e_{2r}\ge 0\\ e_0+e_1+\cdots+e_{2r}=m}}
\zeta^\star(\{2\}^{e_0},3,\{2\}^{e_1},1,\{2\}^{e_2},3,\ldots,
3,\{2\}^{e_{2r-1}},1,\{2\}^{e_{2r}})\\
=\sum_{\substack{2i+k+u=2r\\ j+l+v=m}}
(-1)^{j+k}{k+l\choose k}{u+v\choose u}
{2i+j\choose j}\frac{\beta_{k+l}\beta_{u+v}\pi^{4r+2m}}{(2i+1)(4i+2j+1)!},
\end{multline}
where $\beta_n=(-1)^{n}(2-2^{2n})B_{2n}/(2n)!.$
It is possible to modify our proof of Theorem~\ref{thm:ITTWConj} to give this more quantified version.
\end{rem}

\section{More Conjetures of Imatomi et al.}\label{sec:ITTWConj2}
The following results were first conjectured by Imatomi et al. \cite[Conjecture~4.5]{ITTW2009}.
\begin{thm}\label{thm:ITTWConj2}
Let $m$ and $n$ be two nonnegative integers.
\begin{enumerate}
\item[\upshape (i)]
 We have
\begin{equation*}
 \zeta^\star(\{2\}^n,3,\{2\}^m,1)+\zeta^\star(\{2\}^m,3,\{2\}^n,1)
    =\zeta^\star(\{2\}^{n+1}) \zeta^\star(\{2\}^{m+1}).
\end{equation*}

\item[\upshape (ii)] We have
\begin{equation*}
(2n+1) \zeta^\star(\{3,1\}^n,2)=\sum_{j+k=n} \zeta^\star(\{3,1\}^j)\zeta^\star(\{2\}^{2k+1}).
\end{equation*}

\item[\upshape (iii)] If $n\geq 1$ then we have
\begin{multline*}
\sum_{\begin{subarray}{c} e_1+e_2+\cdots +e_{2n}=1 \\ e_1,e_2,\ldots,e_{2n} \geq 0 \end{subarray}} \zeta^\star(\{2\}^{e_1}, 3, \{2\}^{e_2},1,\ldots, \{2\}^{e_{2n-1}},3, \{2\}^{e_{2n}},1)\\
=\sum_{j+k=n-1} \zeta^\star(\{3,1\}^j,2) \zeta^\star(\{2\}^{2k+2}).
\end{multline*}
\end{enumerate}
\end{thm}
\begin{proof}
\underline{(i)}. This follows immediately from \eqref{equ:Zlobin} and \eqref{equ:MZSV2321r=1}.

\medskip
\underline{(ii)}. We notice that by taking $a_i=b_i=0$ and $c_i=3$ for all $i\le r=j$ in
Theorem~\ref{thm:MHS2c21&212c21}{\rm ($\oll{\bt\hyf c\hyf\bt\hyf1}$)} we get
\begin{equation*}
     \zeta^\star(\{3,1\}^j) =\sum_{\bfp_{2j}\in \Pi(\{\ol{2}\}^{2j})} 2^{\ell(\bfp_{2j})}\zeta(\bfp_{2j}).
\end{equation*}
All of the components $a_j$ of $\ol{2}\circ \cdots \circ \ol{2}$ must satisfy the following sign rule:
\begin{equation}\label{equ:signRule}
    a_j>0 \text{ if and only if }4|a_j.
\end{equation}
On the other hand, by Theorem~\ref{thm:MHS2c212&212c212}{\rm ($\oll{\bt\hyf c\hyf\bt\hyf1}\hyf\bt$)}
we have
\begin{equation*}
     \zeta^\star(\{3,1\}^n,2) =\sum_{\bfp_{2n+1}\in \Pi(\{\ol{2}\}^{2n+1})}  2^{\ell(\bfp_{2n+1})}\zeta(\bfp_{2n+1}).
\end{equation*}
Hence by \eqref{equ:Zlobin} we need to show that
\begin{equation}\label{equ:ITTWmiddlestep}
\sum_{\bfp_{2n+1}\in \Pi(\{\ol{2}\}^{2n+1})}  2^{\ell(\bfp_{2n+1})}\zeta(\bfp_{2n+1})
=\sum_{j=0}^n \sum_{\bfp_{2j}\in \Pi(\{\ol{2}\}^{2j})} 2^{\ell(\bfp_{2j})}\zeta(\bfp_{2j})\cdot 2\zeta(\ol{4(n-j)+2}).
\end{equation}
Suppose an index in $\bfp_{2j}$ has length $t+1$ ($0\le t\le 2n$) given as
$$(a_1,\dots,a_{t+1}), \quad a_i\in \Z^* \ \forall i=1,\dots,t+1.$$
We now show that there are exactly $2^t(2n+1)$ copies of such term produced
by stuffle product on the right hand side
of \eqref{equ:ITTWmiddlestep}. Indeed, for each $i=1,\dots,t+1$ the entry $a_i$ has two possibilities:

(1). $a_i=4b_i>0$. Then for each $k=1,\dots,b_i$ we may produce such a term on the right hand side
of \eqref{equ:ITTWmiddlestep} by stuffing
$$ 2^t \zeta(a_1,\dots,a_{i-1}, \ol{4k-2},a_{i+1}, \dots,a_{t+1})$$
from $\bfp_{2j}$ having length $t+1$ with the term $2\zeta(\ol{4(b_i-k)+2})$ at the right end of
\eqref{equ:ITTWmiddlestep}. Notice no shuffle is possible since $4(n-j)+2$ is not a multiple of $4$.
Hence these contribute to $2^{t+1} b_i=2^{t-1} a_i$ copies of $\zeta(a_1,\dots,a_{t+1})$.

(2). $a_i=\ol{4b_i+2}$. Then for each $k=1,\dots,b_i$ we may produce such a term on the right hand side
of \eqref{equ:ITTWmiddlestep} by stuffing
$$ 2^t \zeta(a_1,\dots,a_{i-1}, \ol{4k},a_{i+1}, \dots,a_{t+1})$$
from $\bfp_{2j}$ having length $t+1$ with the term $2\zeta(\ol{4(b_i-k)+2})$ at the right end of
\eqref{equ:ITTWmiddlestep}. Further, there is exactly one possible shuffle given by
 $$ 2^{t-1} \zeta(a_1,\dots,a_{i-1}, a_{i+1}, \dots,a_{t+1}) \sha \big\{2\zeta(\ol{4b_i+2})\big\},$$
since the index $(a_1,\dots,a_{i-1}, a_{i+1}, \dots,a_{t+1})$ has only length $t$.
Altogether these produce $2^{t+1} b_i+2^t=2^{t-1} |a_i|$ copies of $\zeta(a_1,\dots,a_{t+1})$.

By combining (1) and (2) we see that the right hand side of \eqref{equ:ITTWmiddlestep} produces
exactly
$$\sum_{i=1}^{t+1}2^{t-1} |a_i|=2^{t-1} \cdot|(a_1,\dots,a_{t+1})|=2^t(2n+1)$$
copies of $\zeta(a_1,\dots,a_{t+1})$ since the weight is $4n+2$. This proves (ii).

\medskip
 \underline{(iii)}. We use the same analysis as above and see that we need to prove the following identity:
\begin{equation}\label{equ:ITTWmiddlestep2}
\sum_{\bfq_{2n}} 2^{\ell(\bfq_{2n})}\zeta(\bfq_{2n})
=\sum_{j=0}^{n-1} \sum_{\bfp_{2j+1}\in \Pi(\{\ol{2}\}^{2j+1})}  2^{\ell(\bfp_{2j+1})}\zeta(\bfp_{2j+1})\cdot 2\zeta(\ol{4(n-j)}),
\end{equation}
where $\bfq_{2n}$ runs through all indices of the form
$A_1\circ\cdots\circ A_{2n}$ with one of the $A_j$'s (say $A_{j_0}$) equal to $\ol{4}$ and all the other $A_j$'s equal to $\ol{2}$. For each choice of
$2^t \zeta(a_1,\dots,a_{t+1})$ with length $t+1$ from the left hand of \eqref{equ:ITTWmiddlestep2}, all but one of the argument components $a_1,\dots,a_{t+1}$ must satisfy the sign rule \eqref{equ:signRule}. The only exceptional component, say $a_i$, must involve a merge with the special entry $A_{j_0}=\ol{4}$.
Now there are two possibilities:

(1). $a_i=4b_i+2>0$. Then for each $k=0,\dots,b_i-1$ we may produce such a term on the right hand side
of \eqref{equ:ITTWmiddlestep2} by stuffing
$$ 2^t \zeta(a_1,\dots,a_{i-1}, \ol{4k+2},a_{i+1}, \dots,a_{t+1})$$
from $\bfp_{2j+1}$ having length $t+1$ with the term $2\zeta(\ol{4(b_i-k)})$ at the right end of
\eqref{equ:ITTWmiddlestep2}. Notice no shuffle is possible since $4(n-j)$ is a multiple of $4$.
Hence these contribute to $2^{t+1} b_i$ copies of $\zeta(a_1,\dots,a_{t+1})$.
On the left hand side, such a term must be produced by setting all $2b_i-1$ consecutive $\circ$'s
around $A_{j_0}=\ol{4}$ to $\oplus$:
\begin{equation*}
     \dots, \underset{\text{$2b_i$ entries}}{\underbrace{A_{i}\oplus A_{i+1}\oplus
     \cdots \oplus A_{j_0}\oplus \cdots \oplus A_{\ell}}},\dots.
\end{equation*}
But $A_{j_0}$ can be at any one of the $2b_i$ possible positions, thus
producing $2^{t+1} b_i$ copies of $\zeta(a_1,\dots,a_{t+1})$ which match exactly the right hand side of \eqref{equ:ITTWmiddlestep2}.

(2). $a_i=\ol{4b_i}$. Then for each $k=1,\dots,b_i-1$ we may produce such a term on the right hand side
of \eqref{equ:ITTWmiddlestep2} by stuffing
$$ 2^t \zeta(a_1,\dots,a_{i-1},4k,a_{i+1}, \dots,a_{t+1})$$
from $\bfp_{2j}$ having length $t+1$ with the term $2\zeta(\ol{4(b_i-k)})$ at the right end of
\eqref{equ:ITTWmiddlestep2}. Further, there is exactly one possible shuffle given by
 $$ 2^{t-1} \zeta(a_1,\dots,a_{i-1}, a_{i+1}, \dots,a_{t+1}) \sha \big\{2\zeta(\ol{4b_i})\big\}.$$
Hence these contribute to $2^{t+1} (b_i-1)+2^t=2^t(2b_i-1)$ copies of $\zeta(a_1,\dots,a_{t+1})$.
Similar to (1), on the left hand side, such a term must be produced by setting all $2b_i-2$ consecutive $\circ$'s
around $A_{j_0}$ to $\oplus$. And $A_{j_0}$ can be at any one of the $2b_i-1$ possible positions, thus
producing $2^t(2b_i-1)$ copies of $\zeta(a_1,\dots,a_{t+1})$ which match exactly the right hand side of \eqref{equ:ITTWmiddlestep2}.

This concludes the proof of theorem.
\end{proof}

Note that Theorem~\ref{thm:ITTWConj2}(i) is the more precise version of
the $n=1$ case of Theorem~\ref{thm:ITTWConj}(i). And Theorem~\ref{thm:ITTWConj2}(iii)
can be written more compactly as
\begin{equation*}
 \zeta^\star \big(\{2\}\sha \{3,1\}^n \big)
=\sum_{k=0}^n \zeta^\star(\{3,1\}^{n-k},2) \zeta^\star(\{2\}^{2k}),
\end{equation*}
which is the more precise version of
the $m=1$ case of the following result of Kondo et al.\ \cite{KST2012}:
For all nonnegative integers $m$ and $n$ we have
\begin{equation*}
 \zeta^\star(\{2\}^m\sha\{3,1\}^n)\in \Q\pi^{2m+4n}.
\end{equation*}
The case $m=0$ case has the following precise formulation
by Muneta \cite{Muneta2008}:
\begin{equation*}
\zeta^\star\big(\{3,1\}^{n})
=\sum_{i=0}^{n} \bigg\{ \frac{2}{(4i+2)!}
  \sum_{\begin{subarray}{c} n_0 + n_1 = 2(n-i) \\ n_0, n_1 \geq 0 \end{subarray} }
  (-1)^{n_1} \frac{ (2^{2n_0}-2) B_{2n_0} } {(2n_0)!}
  \frac{ (2^{2n_1}-2) B_{2n_1} } {(2n_1)!} \bigg\} {\pi}^{4n}.
\end{equation*}
Muneta also found precise form in case $m=1$. Of course, these are all special cases
of Yamamoto's general formula \eqref{equ:YamamotoThm}.

\section{MHS: $\oll{\bfone\hyf c}\hyf\bfone$ formula}\label{sec:MHS1c&1c1}
In this section we turn to MHS of the type $\oll{\bfone\hyf c}\hyf\bfone$ where the trailing $\bfone$ may be vacuous and the $c$'s may be any positive integers (which is different from the requirement $c\ge 3$ in the previous sections). The corresponding MZSVs diverge when the leading $\bfone$ is non-empty, however, in a sequel to this paper we will study the congruence properties of MHS where the results of this section will be utilized.

The following theorem generalizes \cite[Theorem~2.2]{HessamiPilehrood2Ta2013}.
For $\bfs=(s_1,\dots,s_m)\in(\Z^*)^m$ we define
\begin{equation*}
  \calh_n(\bfs):= \sum_{n\ge k_1>\cdots >k_m\ge 1} \binom{n}{k_1}\,
 \prod_{j=1}^m  \frac{\sgn(s_j)^{k_j} }{k_j^{|s_j|}}
 =  \sum_{k=1}^n \frac{\sgn(s_1)^{k}}{k^{|s_1|}} \binom{n}{k}\, H_{k-1}(s_2,\dots,s_m).
\end{equation*}

\begin{thm} \label{thm:MHS1c1}
Let $r\in \N$ and $\bfs=(\{1\}^{a_1},c_1,\dots,\{1\}^{a_r},c_r,\{1\}^t)$
where $t,a_j\in\N_0$ and  $c_j\in\N$ for all $j\ge 1$. Then
\begin{equation}\label{equ:1c}
 H_n^{\star}(\bfs)=-\sum_{\bfp\in \Pi(\ol{a_1+1},\,  \{1\}^{c_1-2},\,  a_2+2,\,  \ldots,\,
                             a_r+2,\,  \{1\}^{c_r-2},\, t+1)}   \calh_n(\bfp).
\end{equation}
\end{thm}

\begin{proof}
When $n=1$ it is clear that both sides in \eqref{equ:1c} are equal to 1.
We proceed by induction on $n+r$. By definition
\begin{align*}
    H^\star_n(\bfs)= &\sum_{l=0}^{a_1} \frac{1}{n^{a_1-l}}H^\star_{n-1}( \{1\}^l,c_1,\ldots,\{1\}^{a_r},c_r,\{1\}^t) \\
    +& \frac{1}{n^{a_1+c_1}}H^\star_n(\{1\}^{a_2},c_2,\ldots,\{1\}^{a_r},c_r,\{1\}^t).
\end{align*}
For ease of reading we define the following index sets: for any composition $\bfv$ of integers
\begin{align*}
    I(\bfv)= &\Pi(\bfv,\,  \{1\}^{c_1-2},\,  a_2+2,\,  \{1\}^{c_2-2},\,   \ldots,\, a_r+2,\,  \{1\}^{c_r-2},\, t+1),\\
    J =&\Pi(\ol{a_2+1},\,  \{1\}^{c_2-2},\,  a_3+2,\,  \ldots,\,
                             a_r+2,\,  \{1\}^{c_r-2},\, t+1).
\end{align*}
By induction assumption
\begin{multline*}
    H^\star_n(\bfs)=-\sum_{l=0}^{a_1} \frac{1}{n^{a_1-l}}
        \sum_{\bfq\in I(\ol{l+1})}   \calh_{n-1}(\bfq)
    -\frac{1}{n^{a_1+c_1}} \sum_{\bfp\in J}   \calh_n(\bfp) \\
=-\sum_{(q_1,\dots,q_m)\in I(\ol{1})}
 \sum_{l=0}^{a_1} \frac{1}{n^{a_1-l}}  \sum_{k=1}^{n-1} \frac{\sgn(q_1)^k}{k^{l+|q_1|} } \binom{n-1}{k}\,
 H_{k-1}(q_2,\dots,q_m)
 -\frac{1}{n^{a_1+c_1}} \sum_{\bfp\in J}   \calh_n(\bfp).
\end{multline*}
By changing the order of summations and using the identity
\begin{equation*}
\sum_{l=0}^{a_1} \left(\frac{n}{k}\right)^{l}=\frac{1}{k^{a_1}}\cdot \frac{n^{a_1+1}-k^{a_1+1}}{n-k}
\end{equation*}
we see easily that
\begin{align*}
  H^\star_n(\bfs)=&
 -\sum_{(q_1,\dots,q_m)\in I(\ol{1})}  \sum_{k=1}^n \Big(1-\frac{k^{a_1+1}}{n^{a_1+1}} \Big)
 \frac{\sgn(q_1)^k}{k^{a_1+|q_1|} } \binom{n}{k}\,  H_{k-1}(q_2,\dots,q_m) \\
&-\frac{1}{n^{a_1+c_1}} \sum_{\bfp\in J}   \calh_n(\bfp).
\end{align*}
Observe that the index set
\begin{equation}\label{equ:myIndexSetFinite}
 I(\ol{1}) =\bigcup_{(p_1,\dots,p_m)\in J } \Big\{\big( \ol{1} \circ \underbrace{1 \circ \dots \circ 1}_{c_1-2 \text{ times}}
\circ (p_1+1),p_2,\dots,p_m \big)  \Big\}.
\end{equation}
For each $(\ol{p_1},p_2\dots,p_m)$ we can partition the set \eqref{equ:myIndexSetFinite} into the following subsets:
\begin{alignat*}{4}
 \Big\{ \big(\ol{c_1+p_1},\bfv \big)\Big\} \cup
 \Big\{\big(\ol{j+1},\bfy,i+p_1,\bfv\big)\Big\}, \quad
 \bfv=(p_2,\dots,p_m),
\end{alignat*}
for $i\ge 1$, $j\ge 0$ and positive compositions $\bfy$ with $i+j+|\bfy|=c_1-1$.
Thus it suffices to prove that
\begin{equation}\label{equ:1cusePTLemma2.2a}
 \sum_{\substack{\ i+j+|\bfy|=c_1-1, \\ i\ge 1,j\ge 0}}
\sum_{k=1}^n\frac{H_{k-1}(\bfy,i+p_1) \binom{n}{k}} {(-1)^k k^{j} }
=\frac{1}{n^{c_1-1}} \sum_{k=1}^n\frac{(-1)^{k}\binom{n}{k}} {k^{p_1} }
-\sum_{k=1}^n\frac{  (-1)^{k} \binom{n}{k}} {k^{c_1-1+p_1} }.
\end{equation}
Equation \eqref{equ:1cusePTLemma2.2a} follows from \eqref{lem:combinatorial}
of Lemma~\ref{lem:PTlemma2.2} when $m=1$, $A^{(1)}_{n,k}$ as in Remark~\ref{rem:choiceAm},
$c=c_1-1,a=p_1$, $\bfx=(\bfy,i+p_1)$ and $\bfv=\emptyset$.
This completes the proof of our theorem.
\end{proof}

For example, when $r=1$ we recover \cite[Theorem~2.2]{HessamiPilehrood2Ta2013} and when $r=2$ we get
for all $a_1,a_2,t\in \N_0$ and $c_1,c_2\in \N$
{\allowdisplaybreaks
\begin{multline*}
   H^\star_n(\{1\}^{a_1},c_1,\{1\}^{a_2},c_2,\{1\}^{t})
 =- \sum_{k=1}^n  \frac{(-1)^k\binom{n}{k} }{k^{a_1+c_1+a_2+c_2+t}}  \\
  -\sum_{\substack{\ i_2+j_2+|\bfx_2|=c_2, \\ i_2,j_2\geq 1}}
  \sum_{k=1}^n \frac{H_{k-1}(\bfx_2,i_2+t) \binom{n}{k}} {(-1)^k k^{a_1+c_1+a_2+j_2}}
-\sum_{\substack{\ i_1+j_1+|\bfx_1|=c_1, \\ i_1,j_1\geq 1}}
  \sum_{k=1}^n \frac{H_{k-1}(\bfx_1, i_1+a_2+c_2+t) \binom{n}{k}} {(-1)^k k^{a_1+j_1}} \\
  -\sum_{\substack{\ i_\ga+j_\ga+|\bfx_\ga|=c_\ga, \\ i_\ga,j_\ga\geq 1,\ \ga=1,2}}
  \sum_{k=1}^n \frac{H_{k-1}(\bfx_1,i_1+a_2+j_2,\bfx_2,i_2+t) \binom{n}{k}} { (-1)^kk^{a_1+j_1}}.
\end{multline*}
}

\section{Concluding Remarks}
There are many recent studies on MZVs, MZSVs and even their $q$-analogs.
Most of the MZSV relations in \cite{Igarashi2011} and \cite{IKOO2011} involving special types of arguments like ours in this paper can be proved in a more straight-forward manner using our results. However, it seems that the techniques contained here are hard to generalize to deal with MZVs even though these two types of values are extremely closely related from the point of view of their algebraic structures (see \cite{HoffmanIh2012,IKOO2011,Muneta2009a,TanakaWa2010}). Such a generalization should help us resolve more conjectures such as those listed in \cite[\S7.2]{BBBL1998}.

There are three more directions of research that should be of great interest. One is a theory generalizing the MHS identities obtained in this paper to truly alternating ones. We are aware of only one such instance. Setting $x=0$ and $x=1$ in \cite[Lemma~5.4]{TaurasoZh2010} we get
\begin{equation*}
H^\star_n(\{1\}^a,\ol{1})=\sum_{k=1}^n \frac{(2^k-1)(-1)^k}{k^{a+1}}\binom{n}{k}\quad \forall a\in\N_0.
\end{equation*}
Another direction is to establish a corresponding theory for multiple $q$-zeta values \cite{Bradley2005,Zhao2007c}.
Initial computations show it is quite a promising project, see \cite{HessamiPilehrood2013,HessamiPilehrood2Zhao}.

As for the third direction we notice that the many MHS identities proved in this paper can be used not only to derive MZSV identities but also to prove many congruences of MHS. This idea has already been carried out in \cite{HessamiPilehrood2Ta2013} to prove one of our conjectures from \cite{Zhao2008a}. In general, these congruences should shed more light on the unsolved \cite[Conjecture 2.6]{Zhao2011c} and the conjectures at the end of \cite{Zhao2008a}. This will be done in a sequel to this paper.

\end{document}